\documentclass[a4paper,12pt]{amsart}

\usepackage[english]{babel}
\usepackage{tikz}
\usetikzlibrary{positioning,arrows}

\usepackage{amsmath,amssymb,amsthm}

\usepackage{enumerate}

\usepackage{verbatim}

\newtheorem{theorem}{Theorem}[section]
\newtheorem{thmx}{Theorem}

\newtheorem{proposition}[theorem]{Proposition}
\newtheorem{lemma}[theorem]{Lemma}
\newtheorem{corollary}[theorem]{Corollary}

\theoremstyle{definition}
\newtheorem{definition}[theorem]{Definition}

\theoremstyle{remark}
\newtheorem{example}[theorem]{Example}				
\newtheorem{remark}[theorem]{Remark}

\renewcommand{\Im}{\operatorname{Im}}

\newcommand{\Aut}{\operatorname{Aut}}

\newcommand{\Cay}{{\rm Cay}}
\newcommand{\supp}{\operatorname{supp}}
\newcommand{\Diag}{\operatorname{Diag}}
\newcommand{\Adj}{{\rm A}}

\newcommand{\A}{\mathcal{A}}

\usepackage[pdfpagelabels]{hyperref}

\usepackage{colortbl}
\DefineNamedColor{named}{RoyalBlue}     {cmyk}{1,0.50,0,0}
\DefineNamedColor{named}{BrickRed}      {cmyk}{0,0.89,0.94,0.28}

\AtBeginDocument{%
   \def\MR#1{}
}

\begin{document}

\title{Automorphism groups of Cayley evolution algebras}
\author{C. Costoya}
\address{CITIC, CITMAGA, Departamento de Ciencias de la Computaci{\'o}n y Tecnolog{\'\i}as de la Informaci{\'o}n,
Universidade da Coru{\~n}a, 15071-A Coru{\~n}a, Spain.}
\email{cristina.costoya@udc.es}

\author{V. Mu{\~n}oz}
\address{Departamento de \'Algebra, Geometr{\'\i}a y Topolog{\'\i}a, Universidad de M{\'a}laga, 29071-M{\'a}laga, Spain}
\email{vicente.munoz@ucm.es}

\author{A. Tocino}
\address{Departamento de Matem\'atica Aplicada, Universidad de M{\'a}laga, 29071-M{\'a}laga, Spain}
\email{alicia.tocino@uma.es}

\author{A. Viruel}
\address{Departamento de \'Algebra, Geometr{\'\i}a y Topolog{\'\i}a, Universidad de M{\'a}laga, 29071-M{\'a}laga, Spain}
\email{viruel@uma.es}


\begin{abstract}
In this paper we introduce a new species of evolution algebras that we call Cayley evolution algebras.
We show that if  a field  $\Bbbk$ contains sufficiently many elements (for example if $\Bbbk $ is infinite) then every finite group 
$G$ is isomorphic to $\Aut (X)$ where $X$ is a finite-dimensional absolutely simple Cayley evolution $\Bbbk$-algebra. 
\end{abstract}

\maketitle

\section{Introduction}

The question of whether any finite group may be realised as the full automorphism group of a finite-dimensional simple algebra (over an algebraically closed field $\Bbbk$ of characteristic zero) was raised by Popov in \cite{Popov}. Years later,  in the celebrated paper  by Gordeev and Popov \cite{PopovAnnals},  a positive answer was given  in a more general setting:  if $\Bbbk$ is a field containing sufficiently many elements, every linear algebraic $\Bbbk$-group is isomorphic to the full automorphism group of a finite-dimensional simple $\Bbbk$-algebra (which is neither associative nor commutative).  

Using techniques from Rational Homotopy Theory,  in   \cite[Proposition 4.2, Corollary 4.11(2)]{cmv3} it was shown that every group $G$ is the full automorphism group of an infinite number of non-isomorphic differential graded $\mathbb Q$-algebras $M_n,\; n \geq 1$ (that are associative and commutative in the graded sense). Moreover, if the group $G$ is finite, then for every integer $n \geq 1$, $M_n$ is an elliptic algebra, which means that it is finitely generated with finite-dimensional cohomology over $\mathbb Q.$  

In this paper, we place ourselves in the framework of evolution algebras, which are commutative and non-associative algebras. 
Recall that a  $\Bbbk$-algebra $X$ is called absolutely simple, if for every field extension $\mathbb{F}/\Bbbk$, the  $\mathbb{F}$-algebra $X_{\mathbb F}=X\otimes_\Bbbk {\mathbb F}$ is simple. Then our main result is:

\begin{thmx}\label{thm:main}
	Let $G$ be a finite group and let $\Bbbk$ be a not necessarily finite field with multiplicative group of order $|\Bbbk^*| \geq 2|G|$. Then,  there exists an absolutely simple evolution $\Bbbk$-algebra $X$ such that $\Aut(X_{\mathbb F})\cong G$ for every  field extension $\mathbb{F}/\Bbbk$.
\end{thmx}

The proof of Theorem \ref{thm:main} is postponed to Section \ref{sec:cayleygroupalgebras} and relies on the \emph{good} properties of Cayley evolution algebras, which are new objects introduced in this work (see Definition \ref{def:cay}). They are constructed  out of any finite-dimensional associative $\Bbbk$-algebra $\mathcal A$, $\Bbbk$ an arbitrary field. In the special case that  $\mathcal A= \Bbbk[G]$,  a group algebra with $G$ finite group,  the associated Cayley evolution algebras will be closely related to classical Cayley graphs.

We end this section by pointing out that our main result is an improvement of our previous result \cite[Theorem 1.1]{CLTV} where finite groups were realised by regular, but not simple (see Remark \ref{notsimple}), evolution algebras (see also \cite{chinos} for an independent proof in char $\Bbbk =0$).


\section{Background on evolution algebras}\label{sec:evolution}

We present here the definitions on evolution algebras that are needed in the following sections. We mainly follow the notation in \cite{Tesis-Yolanda}, although the reader is encouraged to also check \cite{Tian}, a classical reference on the subject. Other properties of evolution algebras appear in \cite{extra3,extra1,extra2}.

Let $\Bbbk$ be a field. 
An automorphism of $\Bbbk$-algebra is a linear isomorphism which commutes with the multiplication
of the algebra. In this work we consider only finite-dimensional algebras.
We now recall the definition of an evolution algebra \cite[Definition 1]{Tian}, \cite[Definitions 1.2.1]{Tesis-Yolanda}:

\begin{definition}\label{def:evolution}
Let 
$X$ be a $\Bbbk$-algebra. We say that $X$ is an \textit{evolution $\Bbbk$-algebra} if  $X$ admits a  basis $B = \{b_i\, |\, i \in \Lambda\}$ such that $b_i b_j = 0$ for $i\neq j$. Such a basis $B$ is called a \textit{natural basis}.
\end{definition}

Observe that given an evolution $\Bbbk$-algebra $X$, it may admit more than one natural basis (see for example \ \cite[Example 1.6.3]{Tesis-Yolanda}). On the other hand, for a given $\Bbbk$-vector space $X$ spanned by a basis $B = \{b_i\, |\, i \in \Lambda\}$,  an evolution  $\Bbbk$-algebra with natural basis $B$ is completely determined by just giving $b_i^2:=b_i b_i\in X$ as  a linear combination of elements in $B$, for every $i\in\Lambda$, and declaring  $b_ib_j=0$ for $i \neq j$. This motivates the following definition \cite[Definition 1.2.1]{Tesis-Yolanda}:

\begin{definition}\label{def:structure}
	Let $X$ be an evolution algebra with a fixed natural basis $B$. The scalars $\omega_{ki} \in\Bbbk$ such that
	$b_i^2:=b_i b_i=\sum_{k\in\Lambda}\omega_{ki}b_k$ are called the \textit{structure constants} of $X$ relative to $B$, and the matrix $M_B(X) := (\omega_{ki})$ is said to be the \textit{structure matrix} of $X$ relative to $B$.
\end{definition}

\begin{definition}\label{def:regular}
	An evolution  $\Bbbk$-algebra $X$ is called \textit{regular} (or perfect, or idempotent) if  $X=X^2$.	An evolution algebra $X$ is called \textit{simple} if $X^2\neq 0$ and $0$ is the only proper ideal.
\end{definition}

	Notice  that an evolution algebra $X$ is regular if and only if for any natural basis $B$, the structure matrix $M_B(X)$ is a regular (or invertible) matrix. It is also clear that if $X$ simple then $X$ is regular, but the converse is not always true (see Proposition \ref{prop:Cay simple_algebra}).

Every evolution algebra has a directed graph and a weighted (or coloured) graph attached depending on the chosen natural basis \cite[Definition 15]{Tian}, \cite[Definition 2.2]{Elduque-Labra-2015}:

\begin{definition}\label{def:digraph}
Let $X$ be an evolution $\Bbbk$-algebra, and $B=\{b_1,\ldots,b_n\}$ its natural basis. 
The directed graph $\Gamma(X,B)=(V,E)$ with set of vertices $V=\{1,\ldots,n\}$ and set of edges $E=\{(i,j)\in V\times V: \omega_{ij}\ne 0\}$ is called the directed graph attached to the evolution algebra $X$ relative to $B$. The directed graph $\Gamma^w(X,B)=(V,E,w)$ with $\Gamma(X,B)=(V,E)$ and weight function $w\colon E\to \Bbbk$,  given by $w(i,j)=\omega_{ij}$, is called the weighted graph attached to the evolution algebra $X$ relative to $B$.
\end{definition}

\section{The Cayley evolution algebra of a $\Bbbk$-algebra}\label{sec:general}

We now explain how, to a given associative $\Bbbk$-algebra $\A$, we can associate an evolution $\Bbbk$-algebra that takes into account the inner product of $\A$:

\begin{definition}\label{def:cay}
Let $\A$ be a finite-dimensional associative $\Bbbk$-algebra with basis $B=\{b_i\,\colon\,i\in\Lambda\}$ together with a function $f:B\rightarrow\Bbbk$ 
(a set-theoretical map). We define the \textit{Cayley evolution  $\Bbbk$-algebra associated to $f$},  that we denote by $\Cay(f)$, as the evolution $\Bbbk$-algebra  with $\Cay(f)\cong \A$, as  a $\Bbbk$-vector space, furnished with natural basis $B$ and structure constants given by 
$b_i\cdot b_i=\sum_{b_j\in B}f(b_j)b_ib_j$, where  $x\cdot y$ denotes the product in $\Cay(f)$ and $xy$ the inner product in $\A$.
\end{definition}

\begin{remark}\label{rmk:cayley_con_vector}
Observe that fixing a function $f:B\rightarrow\Bbbk$ is equivalent to choosing a vector $w=\sum f(b_i)b_i$. Therefore,  the Cayley evolution algebra associated to $f$ could also be defined as the Cayley evolution algebra associated to a vector $w\in\A$ with $\Cay(w)\cong \A$, as $\Bbbk-$vector space, furnished with natural basis $B$ and structure constants given by $b_i\cdot b_i=\sum_{b_j\in B}w^*(b_j)b_ib_j$, where $w^*$ is the dual of $w$.
\end{remark}

\begin{remark}\label{rmk:diferentes_cayley}
Given a finite-dimensional associative $\Bbbk$-algebra $\A$ with basis $B$, the isomorphism type of the evolution $\Bbbk$-algebra $\Cay(f)$ strongly depends on the choice of $f:B\rightarrow\Bbbk$. For example, if $\A$ is  a non-degenerate evolution algebra with natural basis $B$, then the constant map given by $f_0(b_i)=0$ converts $\Cay(f_0)$ into a degenerated evolution algebra, whereas the constant map given by $f_1(b_i)=1$ transforms $\Cay(f_1)\cong\A$. A more elaborated example is constructed  in Example \ref{ex:hertwek} to illustrate  how the isomorphism type of the Cayley evolution algebra depends on the choice of the basis. 
\end{remark}


From now on, let us fix a unital finite-dimensional associative $\Bbbk$-algebra $\A$, with basis $B$, and a function $f:B\rightarrow\Bbbk$.  We now prove several properties of the Cayley evolution $\Bbbk$-algebra associated to $f$, $\Cay(f)$. First we show how units in $\A$ give rise to automorphisms of $\Cay(f)$.

\begin{proposition}\label{prop:units_vs_aut}
If $x\in\A$ is a unit such that $xB=B$ then the left-multiplication by $x$ induces $\psi_x\in\Aut\big(\Cay(f)\big)$. If moreover $B$ consists of units in $\A$, then  $\psi_x$ acts freely on $B$ for $x \neq 1$.
\end{proposition}
\begin{proof}
	Given a unit $x\in\A$, we construct the map $\psi_x:\Cay(f)\rightarrow \Cay(f)$ given by $\psi_x(a)=xa$. Observe that $\psi_x$ maps the natural basis $B$ to itself since $\psi_x(B)=xB=B$. Hence it is an automorphism of $\Cay(f)$  as $\Bbbk$--vector space.  Now, 	to check that $\psi_x$ is an automorphism of evolution algebras, on the one hand we have that
$\psi_x(b_i\cdot b_i)=\psi_x(\sum_{b_j\in B}f(b_j)b_ib_j)=\sum_{b_j\in B}f(b_j)x(b_ib_j)=\sum_{b_j\in B}f(b_j)(x b_i)b_j= \sum_{b_j\in B}f(b_j)\psi_x(b_i)b_j=\psi_x(b_i)\cdot \psi_x(b_i)$. On the other hand, for $b_i \neq b_j,$  we have that $\psi_x(b_i\cdot b_j)=\psi_x(0)=x0=0 $ and we are going to prove that $\psi_x(b_i)\cdot \psi_x(b_j)$ is also zero. Suppose that it is not, hence $ \psi_x(b_i)\cdot \psi_x(b_j)=xb_i\cdot xb_j \neq 0$, and since $xb_i$ and  $xb_j $ are elements in the natural basis $B$,  they must be equal. But $x$ is a unit in $\A$, which leads to the contradiction $b_i = b_j$.

Now, suppose that $B$ consists of units in $\A$. If $\psi_x$ fixes an element $b \in B$,  then $\psi_x(b)=xb= b$,   so we obtain that $x=1$ since $b$ is a unit.
\end{proof}

Let us define $\supp(f: B \to \Bbbk):=\{b\in B\colon f(b)\ne 0\}$.

\begin{proposition}\label{prop:Cay simple_algebra}
Suppose that $B$ consists of units in $\A$, and that $\supp(f)$ generates $\A$ as an algebra. If the Cayley evolution $\Bbbk$-algebra associated to $f$, $\Cay(f)$, is regular, then $\Cay(f)$  is also simple.
\end{proposition}
\begin{proof}
	Suppose that $\Cay(f)$ is regular but not simple. Then, according to  \cite[Corollary 4.6]{Cabrera-Siles-Velasco}, the elements in $B$ can be reordered, namely $B=\{b_{1},b_{2},\ldots,b_{r},\ldots,b_n\}$, $1\leq r<n$, such that the structure matrix of $\Cay(f)$ becomes
	$$M_B(\Cay(f))=\left(
	\begin{array}{ccc|ccc}
	* & \ldots  & *  & * & \ldots & *\\
	\vdots & r \times r  & \vdots  & \vdots  & \ddots & \vdots\\
	* & \ldots  & *  & *  & \ldots  & *\\
	\hline
	&  &   & * & \ldots  & *\\
	 & \text{\Large 0}&   & \vdots & \ddots  & \vdots\\
	 & & & * & \ldots  & *\\
	\end{array}
	\right).$$	
Let $S=\supp(f)$ and consider an element $b\in B\smallsetminus\{b_1,\ldots,b_r\}$. Since $S$ generates $\A$ as algebra, and $B$ consists of units in $\A$, the product $b_1^{-1}b=\sum_k \lambda_k \prod_{i=1}^{l_k}s_{i,k}$, where $\lambda_k\in\Bbbk$ and $s_{i,k}\in S$. Then,
$b  = b_1b_1^{-1}b
= b_1\big(\sum_k \lambda_k \prod_{i=1}^{l_k}s_{i,k}\big)
= \sum_k \lambda_k b_1\big(\prod_{i=1}^{l_k}s_{i,k}\big).$

We claim that $b_1\big(\prod_{i=1}^{l_k}s_{i,k}\big) \in Span(\{b_1,\ldots,b_r\})$. Indeed, for $j\in\{1,\ldots,r\}$,
	$$b_j\cdot b_j=\sum_{b_k\in B}f(b_k)b_jb_k=\sum_{s\in S}f(s)b_js,$$
thus,  by the shape of $M_B(\Cay(f))$, we conclude that $b_js\in Span ( \{b_1,\ldots,b_r\} )$ for all $j\in\{1,\ldots,r\}$ and for all $s\in S$. Even more, given $\prod_{i=1}^{l_k}s_{i,k}$, for $s_{i,k}\in S$, an inductive argument shows that  $b_j\big(\prod_{i=1}^{l_k}s_{i,k}\big)\in Span(\{b_1,\ldots,b_r\})$ for all $j\in\{1,\ldots,r\}$. In particular, we obtain that $b \in Span(\{b_1,\ldots,b_r\}).$ This leads to a contradiction as $B$ is a linearly independent set,  hence $\Cay(f)$ is simple.

\end{proof}

\begin{remark}\label{notsimple}
As we mentioned in Introduction,  the evolution algebras obtained in \cite{cmv3} are regular but not simple. Recall that associated to a simple graph $\mathcal G = (V,E)$, we construct an evolution algebra $\mathcal X (G)$ over any field $\Bbbk$ with natural basis $B = \{b_v\,,\, v \in V\} \cup \{b_e \,,\,e \in E\}$ (see \cite[Definition 3.2]{cmv3}). In particular, the multiplication for the elements associated to the vertices of the graph is given by $b_v^2 = b_v,$ for all $v \in V$. Hence, for  $I_r$ the identity matrix of order $r= |V|$, and $I_s$ the identity matrix of order $s=|E|$, the structure matrix of   $\mathcal X (G)$  becomes
	$$M_B(\mathcal X (G))=\left(
	\begin{array}{ccc|ccc}
	  &    &    & * & \ldots & *\\
	  &  I_r &    & \vdots  & \ddots & \vdots\\
	  &    &    & *  & \ldots  & *\\
	\hline
	 & &   &  &   & \\
	 &\text{\large 0}&  & & I_s & \\
	 && & &   & \\
	\end{array}
	\right) $$	
which by \cite[Corollary 4.6]{Cabrera-Siles-Velasco}, implies that $\mathcal X (G)$ is  not simple for $r,s>0$.

\end{remark}

\section{Cayley evolution algebras of group algebras}\label{sec:cayleygroupalgebras}

In this section we refine the results in Section \ref{sec:general}. We fix a finite group $G$ and we consider the group algebra over $\Bbbk$, $\A=\Bbbk[G]$. Now, we take as a basis $B=G$ and a function $f:G\rightarrow\Bbbk.$
In this setting, observe that the structure constants of  $\Cay(f)$ with respect to $G$ can be given by $g\cdot g=\sum_{k\in G}f(g^{-1}k)k$.

\if{}
\begin{remark}\label{remark:Cay_simple}
Observe that Proposition \ref{prop:Cay simple_algebra} can be done starting from a finite group $G$ instead of an algebra. So, if $\Cay(f)$ is regular and $S=\supp(f)$ generates $G$, then $\Cay(f)$ is simple.
\end{remark}
\fi

Our first result is a straightforward consequence of Proposition \ref{prop:units_vs_aut}:
\begin{proposition}\label{prop:faith}
	$G$ acts faithfully on $\Cay(f)$ as a permutation group over its natural basis by left-multiplication. That is, $G\leq \Aut(\Cay(f))$.
\end{proposition}
\begin{proof}
The elements of $G$ are units in $\A=\Bbbk[G]$ and since $G^2=G$, then $GB=B$. Hence by Proposition \ref{prop:units_vs_aut}, $\psi_g \in \Aut\big(\Cay(f)\big).$ Moreover, for any $g,h\in G$, $\psi_g\circ\psi_h=\psi_{gh}$, thus the units $G\subset \A$ produce  a subgroup of automorphisms which is isomorphic to $G$.
\end{proof}

We now recall the definition of Cayley graph \cite[Definition 4-5]{CayleyGraph}.

\begin{definition}\label{def:cayley_graph}
Let $G$ be a finite group, and $S$ be a subset of $G$. The \textit{Cayley graph} of $G$ with respect to $S$ is the directed graph $\Cay(G,S)$ with vertex set $G$ and edges $(g,gs)$ for each $g\in G$ and $s\in S$.
Moreover, we can assign a colour $c_s$  to each edge  $(g,gs)$,  $s\in S$, where $c_s\ne c_{s'}$ if $s\ne s'$. Hence, we obtain an  \textit{edge-coloured directed graph}, the coloured Cayley graph. In this work we will denote it by $\Cay^{cor}(G,S)$.\end{definition}

The Cayley evolution algebras that we have previously introduced are closely related to Cayley graphs. By comparing Definitions \ref{def:digraph} and \ref{def:cayley_graph}  we immediately obtain:

 \begin{lemma}\label{lem:graph_of_cayley_is_cayley}
Let $S=\supp(f)$, and assume $f|_S$ is injective. Then $\Gamma\big(\Cay(f),G\big)=\Cay(G,S)$ as abstract directed graphs while  $\Gamma^w\big(\Cay(f),G\big)=\Cay^{cor}(G,S)$ as edge-coloured directed graphs.
 \end{lemma}

One of the key features of the coloured Cayley graph $\Cay^{cor}(G,S)$ is that whenever $S$ generates $G$, the group $G$ can be represented as the full group of automorphisms of $\Cay^{cor}(G,S)$, that is $G\cong \Aut (\Cay^{cor}(G,S))$ \cite[Theorem 4-8]{CayleyGraph}. In a similar way we prove:

\begin{proposition}\label{th:coprime}
Suppose that  the evolution $\Bbbk$-algebra $\Cay(f)$ is regular, $S=\supp(f)$ generates $G$ and $f|_S$ injective. If $S$ contains two elements of coprime order, then $G \cong \Aut (\Cay(f))$.
\end{proposition}
\begin{proof}
First of all,  we make use of the identification $\Gamma\big(\Cay(f),S\big)=\Cay(G,S)$  as abstract directed graphs, according to Lemma \ref{lem:graph_of_cayley_is_cayley}. Now, since $\Cay(f)$ is regular,  by \cite[Theorem 3.2]{Elduque-Labra-2019}, there exists an exact sequence
	$$1\hookrightarrow \Diag(\Cay(G,S))\rightarrow \Aut (\Cay(f))\rightarrow \Aut (\Cay(G,S)),$$
with  $\Diag(\Cay(G,S)) = \mu_N = \{ k \in \mathbb K \, ; k^N =1\}$ for  $N = 2^{\rm b(\Cay(G, S))}-1 $ 	where  $\operatorname b(\Cay(G, S))$ is the balance of $ \Cay(G,S)$ (see  \cite[Section 2]{Elduque-Labra-2019} and \cite[Theorem 2.7]{Elduque-Labra-2019}).

We first calculate $\Diag(\Cay(G,S))$. Notice that for every $s\in S$,  the graph $\Cay(G,S)$ has a directed cycle of length  the order of the element $s$, $o(s)$. Therefore, if  $s_1,s_2\in S$ are elements such that ${\rm gcd} (o(s_1),o(s_2))=1$, we obtain $$1\leq {\rm b}(\Cay(G,S))= {\rm gcd}\{|{\rm b} (\gamma)|\,:\, \gamma\text{ cycle in} \,\Cay(G,S) \}\leq {\rm gcd} (o(s_1),o(s_2))=1.$$
Hence $N=2^{{\rm b}(\Cay(G,S))-1}=1$, and  $\Diag(\Cay(G,S))=\mu_1=\{1\}$ so, $$\Aut(\Cay(f))\leq \Aut(\Cay(G,S)).$$
Therefore, elements in $\Aut(\Cay(f))$ can be thought as those graph automorphisms of $\Cay(G,S)$ that preserve the structure matrix $M_G(\Cay(f))$ in the following sense. Let $\psi \in \Aut(\Cay(f))$ and $g \in G$ an arbitrary element.
As  $\psi$ preserves the product in $\Cay(f)$,  for  $g\cdot g=\sum_{k\in G}f(g^{-1}k)k$, we have that

\begin{equation}\label{psiprod}\psi (g\cdot g) =\sum_{k\in G}f(g^{-1}k) \psi (k)
\end{equation}
 must coincide with $\psi (g) \cdot \psi (g) = \sum_{k\in G}f( \psi(g)^{-1}k)k$.  Since $\psi $ can be thought as a graph automorphism of $\Cay(G,S)$,   $\psi$ is in particular a permutation of the elements in $G$, which allows us to express
 \begin{equation}\label{prodpsi}
 \psi (g) \cdot \psi (g) = \sum_{k \in G}f( \psi(g)^{-1}\psi(k)) \psi(k).
 \end{equation}
 Comparing both \eqref{psiprod} and \eqref{prodpsi} we deduce that $  f(g^{-1}k)= f( \psi(g)^{-1}\psi(k)) $  for arbitrary $g, k \in G$. In other words, $\psi$ is an automorphism of the weighted graph $\Gamma^w\big(\Cay(f),G\big)$ which, by
 Lemma \ref{lem:graph_of_cayley_is_cayley} coincides with $\Cay^{cor}(G,S)$. That is,  $$\Aut(\Cay(f))\leq \Aut(\Cay^{cor}(G,S))\cong G.$$
	Since $G\leq \Aut(\Cay(f))$ by Proposition \ref{prop:faith}, we conclude that $G \cong \Aut (\Cay(f)) $.
\end{proof}


We prove now that if the field $\Bbbk$ is large enough compared with $|G|$ (for instance, if $\Bbbk$ is infinite),  it is then possible to construct $\Cay(f)$, a Cayley evolution $\Bbbk$-algebra associated to the function $f: G \to \Bbbk$, fulfilling  hypotheses from  Lemma \ref{lem:graph_of_cayley_is_cayley} and  Proposition \ref{th:coprime}.

The following lemma makes the condition ``large enough compared with an integer" precise. Recall that a polynomial $P\in \Bbbk[X_1,\ldots, X_m]$ is said to be homogeneous of degree $n\in\mathbb{N}$, denoted by $\deg(P)=n$, if $P$ is a linear combination of monomials $X_1^{d_1}X_2^{d_2}\cdots X_m^{d_m}$ such that $\sum d_i=n$.

\begin{lemma}\label{lem:large_enough}
Let $\Bbbk$ be a not necessarily finite field with multiplicative group of order $|\Bbbk^*|$, and let $n$ and $m$ be  positive  integers such that $|\Bbbk^*|\geq 2n$ and $n \geq m$. Then, for any homogeneous polynomial of degree $n$ in $m$ variables
$$P=\sum_{i=1}^m \lambda_iX_i^n + Q\in\Bbbk[X_1,\ldots, X_m],$$ with $\lambda_i \neq 0$ for all $i$, and where every monomial in $Q$ involves at least two different variables,  there exists a choice of non-zero pairwise distinct values $k_1,\ldots, k_m\in\Bbbk^*$  such that $P(X_i=k_i: i=1,\ldots, m)\ne 0$.
\end{lemma}
\begin{proof}
Fix $n \geq m$ and $P\in\Bbbk[X_1,\ldots, X_m]$ as in the hypotheses. Since  $|\Bbbk^*| \geq m$,  we can choose $k_1,\ldots, k_{m-1}\in\Bbbk^*$ pairwise distinct values.  Hence, the polynomial $\widetilde{P}(X_m):=P(X_i=k_i: i<m)$ is a non-trivial polynomial of degree $n$, and therefore the equation $\widetilde{P}(X_m)=0$ admits a most $n$  different non-zero solutions. That is, there exist at most $\kappa_1,\ldots,\kappa_n\in\Bbbk^*$ such that  $\widetilde{P}(X_m=\kappa_r)=0$, for $r=1, \ldots, n$. Since $|\Bbbk^*|\geq 2n> m-1+n$, there exists a non-zero value $k_m\in\Bbbk^*$ such that $k_m\ne k_j$ for $j=1,\ldots,m-1$ and $k_m\ne \kappa_r$ for $r=1,\ldots, n$. Then, $k_1,\ldots, k_m\in\Bbbk^*$ are non-zero pairwise distinct values and $P(X_i=k_i: i=1,\ldots, m)\ne 0$.
\end{proof}

\begin{theorem}\label{th:ultimo}
Let $\Bbbk$ be a not necessarily finite field with multiplicative group of order $|\Bbbk^*|$ and let $G$ be a finite group. If $|\Bbbk^*|\geq 2 |G|$,  then for any non-empty $S\subset G$ there exists a function $f \colon G\rightarrow \Bbbk$ such that:
	\begin{enumerate}
		\item\label{th:ultimo_1} $S=\supp(f)$
		\item\label{th:ultimo_2} $f|_S$ is injective
		\item\label{th:ultimo_3} $\Cay(f)$ is regular.
	\end{enumerate}
\end{theorem}
\begin{proof}
For any function $f:G\rightarrow\Bbbk$,  if we denote by $X_g$ the variable such that $X_g=f(g)$,  the structure coefficients of  $\Cay(f)$ can be described by
$$g\cdot g=\sum_{h\in G}f(h)gh=\sum_{k\in G}f(g^{-1}k)k=\sum_{k\in G}X_{g^{-1}k}k, $$ and therefore,  the structure matrix of $\Cay(f)$ becomes
$$M_B(\Cay(f))=(X_{g^{-1}k})_{g,k\in G}\in M_{|G|\times |G|}(\Bbbk),$$
whose determinant  is a homogeneous polynomial $P(X_g\,:\,g\in G)$ of degree $|G|$, that we now describe.
Given $k\in G$, let  $P_k:=P(X_g=0: g\neq k)\in \Bbbk[X_k],$ thus $P=\sum_{k\in G} P_k + Q$ where $Q$ is a homogeneous polynomial in $\Bbbk[X_g : g\in G]$ such that every monomial in $Q$ involves at least two different variables. Since $P$ is homogeneous, $P_k(X_k)=\lambda_k X_k^{|G|}$ for some $\lambda_k\in\Bbbk$. Following the notation of \cite{harary}, we can think of the polynomials $P_k$ as 
$$P_k(X_k)=\det\Big(\Adj \big(\Cay(G,\{k\}), X_k\big)\Big)$$ where $\Adj\big(\Cay(G,\{k\}), X_k\big)$ denotes the variable adjacency matrix of the Cayley graph $\Cay(G,\{k\})$.
And, according to  \cite[Theorem 2]{harary}, as $\Cay(G,\{k\})$ is a directed graph consisting of a total of $\frac{|G|}{o(k)}$ disjoint directed cycles of length $o(k)$, which are also the strong components of $\Cay(G,\{k\})$, we obtain that
$$P_k(X_k)= \Big( \det \big( \Adj \big(o(k)\text{-cycle}, X_k \big) \big) \Big)^{\frac{|G|}{o(k)}}.$$
Since the variable adjacency matrix of a directed $o(k)$-cycle is of the form
$$\left(
	\begin{array}{ccccc}
	0 & X_{k}  & 0  & \ldots & 0\\
	0 & 0  & X_{k}  & \ldots  & 0\\
	\vdots & \vdots  & \vdots  & \ddots  & \vdots\\
	0 & 0  & 0  & \ldots & X_{k}\\
	X_{k} & 0 & 0  & \ldots & 0\\
	\end{array}
	\right) \in M_{o(k) \times o(k)} (\Bbbk)$$
we conclude that $$P_k(X_k)=\Big((-1)^{o(k)+1}X_{k}^{o(k)}\Big)^{\frac{|G|}{o(k)}}= (-1)^{\frac{(o(k)+1)|G|}{o(k)}} X_k^{|G|}.$$

We proceed now to construct, for a given $S \subset G$,  the desired function $f\colon G\to \Bbbk$. Let  $P_S:=P(X_g=0: g\not\in S)\in\Bbbk[X_s:s\in S]$ which is a non-trivial homogeneous polynomial of degree $|G|$ on $|S|$ variables. As $|\Bbbk^*| \geq 2n$, Lemma \ref{lem:large_enough} ensures the existence of non-zero pairwise distinct values $k_s\in\Bbbk^*$, $s\in S$, satisfying that $P_S(X_s=k_s : s\in S)\ne 0$.  With that in mind, let $f\colon G\to \Bbbk$ be defined by
$$
f(g)=\begin{cases}
			k_g, & \text{if $g\in S$}\\
            0, & \text{otherwise}
		 \end{cases}
$$
Notice that since all the values $k_s\in\Bbbk^*$, for $s\in S$, are non-zero and pairwise distinct, $S=\supp(f)$ and $f|_S$ is injective. Finally,
\begin{equation*}
\begin{split}
\det\big(M_B(\Cay(f))\big) & = P(X_g=f(g):g\in G)\\
& = P_S(X_s=k_s:s\in S) \\
& \ne 0
\end{split}
\end{equation*}
hence  $\Cay(f)$ is regular.
\end{proof}

Combining the results previously obtained, we give the proof of our main result.

\begin{proof}[Proof of Theorem \ref{thm:main}]
Let $S$ be a set of generators of $G$ containing coprime order elements. Observe that $S$ is also a set of generators of $\Bbbk[G] $ as an algebra, and such $S$ always exists,  it suffices 
to consider a generating set $S$ containing $1\in S$, for example $S = G$. Considering the function $f:G\rightarrow \Bbbk$ given in Theorem \ref{th:ultimo}, $\Cay(f)$ is regular, $S=\supp(f)$ and $f|_S$ is injective. Moreover, $S$ contains coprime order elements, so by Proposition \ref{th:coprime} we  conclude that $\Aut(\Cay(f))\cong G$.  Moreover, since $S$ generates $\Bbbk[G]$ and $\Cay(f)$ is regular, by Proposition \ref{prop:Cay simple_algebra} we obtain that $\Cay(f)$ is simple. 

We now check that the simple evolution $\Bbbk$-algebra $X =\Cay(f) $ above is indeed absolutely simple, and $\Aut(X_{\mathbb F})\cong G$ for every  field extension $\mathbb{F}/\Bbbk$. In fact, since the function $f_{\mathbb F}\colon G\to {\mathbb F}$ induced by $f$ also satisfies conclusions \eqref{th:ultimo_1},   \eqref{th:ultimo_2}, and  \eqref{th:ultimo_3} in Theorem \ref{th:ultimo},  $\Aut(\Cay(f_{\mathbb F}))\cong G$, and $\Cay(f_{\mathbb F})$ is simple. Finally,  since ${\mathbb F}[G]=\Bbbk[G]\otimes_\Bbbk {\mathbb F}$ then  $\Cay(f_{\mathbb F})=\Cay(f)\otimes_\Bbbk {\mathbb F}=X_{\mathbb F}$, thus  $X_{\mathbb F}$ is simple and $\Aut(X_{\mathbb F})\cong G$.

\end{proof}

We end this paper with some remarks on the non-uniqueness of the simple evolution algebras realizing finite groups.

\begin{remark}\label{rmk:not_unique}
Our construction of the simple evolution $\Bbbk$-algebra $\Cay(f) $ such that $\Aut(\Cay(f))\cong G$, for a fixed finite group $G$ involves several choices: first,  a generator set $S\subset G$ containing coprime order elements, which is not unique if $|G|>1$. Secondly,  a function $f\colon G\to \Bbbk$ such that $\Cay(f)$ is regular, as in the proof of Theorem \ref{th:ultimo}, which is not is unique either since any non-zero scalar multiple of $f\colon G\to \Bbbk$  will also work.  An easy argument on the edge-coloured directed graphs $\Gamma^w\big(\Cay(f),B\big)=\Cay^{cor}(G,S)$ illustrates that both choices lead to non-isomorphic evolution $\Bbbk$-algebras with the same group of automorphism $G$.
\end{remark}

\begin{remark}\label{rmk:bounds}
For any evolution algebra $\Cay(f)$ built upon the group algebra $\A=\Bbbk[G]$ with fixed natural basis $B=G$, and any given $f\colon B\to \Bbbk$,  the following  holds $$G\leq\Aut(\Cay(f))\leq \rm{GL}(\Bbbk, |G|)$$ since $B=G$ is a  group of units basis in  $\A=\Bbbk[G]$ (see Proposition \ref{prop:faith}). The lower bound  $G= \Aut(\Cay(f))$ is obtained when $f$ is as in the proof of Theorem \ref{thm:main} (see also Theorem \ref{th:ultimo}) while the upper bound  $\Aut(\Cay(f))=\rm{GL}(\Bbbk, |G|)$ is obtained when $f= 0$.   Although it is not the purpose of this paper, we could raise  the question of  which  intermediate groups  $H$,   $$G\leq H\leq \rm{GL}(\Bbbk, |G|),$$  can be realized as automorphisms of simple Cayley evolution algebras. That is,   $\Aut(\Cay(f))= H$ for  some function $f\colon G\to \Bbbk$.
\end{remark}
We now highlight that not only $\Aut(\Cay(f))$ depends on how \emph{good} or \emph{bad} the function $f: G \rightarrow \Bbbk$ is, as Remark \ref{rmk:bounds} illustrates, but it also depends on the group of units of $\A=\Bbbk[G]$.  And in fact,  since the isomorphism problem for group rings \cite[Problem 1.1]{Sandling} has a negative answer in general \cite{Hert},  the group of units basis is not unique in a group algebra.  The following example shows that the choice of the group of units basis as natural basis for  $\Cay(f)$ might lead to non-isomorphic Cayley evolution algebras with non-isomorphic group of automorphisms:

\begin{example}\label{ex:hertwek}
Let $G_1$ and $G_2$ be the finite groups of order $n=2^{21}97^{28}$ denoted by  respectively $X$ and $Y$ in \cite[Theorem B]{Hert}. Let $\Bbbk$ be a field (not necessarily finite) with multiplicative group of order $|\Bbbk^*| \geq 2 \cdot 2^{21}97^{28}$ and let
 $\A= \Bbbk[G_1]$ be the group $\Bbbk$-algebra.  By \cite[Theorem B]{Hert},  although $G_1$ and $G_2$  are non-isomorphic, they can both be the group of units basis of $\A$. Hence, following the proof of Theorem \ref{thm:main},  there exist functions $f_1\colon G_1\to \Bbbk$,  and $f_2\colon G_2\to \Bbbk$, such that $\Cay(f_1)$ and  $\Cay(f_2)$ are simple evolution $\Bbbk$-algebras with  $\Aut(\Cay(f_1))\cong G_1 \not \cong G_2 \cong \Aut(\Cay(f_2))$, and so $\Cay(f_1) \not \cong \Cay(f_2)$.  
\end{example}

Finally, it is also natural to ask whether the function $f\colon G\to \Bbbk$ defining $\Cay(f)$ can be chosen to be a significant one, for example a character, or more generally, a  class function, that is $f(g^{-1}kg)=f(k)$ for all $g,k\in G$. We prove the following:

\begin{proposition}\label{prop:class_function}
Let  $G$ be a finite group, $\A=\Bbbk[G]$, and $f\colon G\to \Bbbk$ be a class function. Then $\Aut(\Cay(f))$ contains two subgroups $K_i$, $i=1,2$, such that $K_1\cong G$, $K_2\cong G/Z(G)$ and $K_1\cap K_2=\{1\}$.
\end{proposition}
\begin{proof}
By Proposition \ref{prop:faith},  there exists $K_1\cong G\leq \Aut(\Cay(f))$. Moreover, since the natural basis $B=G$ of $\Cay(f)$ consists of units in $\A=\Bbbk[G]$, Proposition \ref{prop:units_vs_aut} ensures that $K_1$ acts freely on the natural basis of $\Cay(f)$.

We now consider the following morphism $\rho\colon G\to \Aut(\Cay(f))$: for any $h\in G$, let $\rho(h)$ be the linear automorphism of $\Cay(f)=\Bbbk[G]$ given by inner conjugation in $G$, that is, $\rho(h)(g)=hgh^{-1}$. 
Observe that $\rho(h)$ maps the natural basis $B=G$ of $\Cay(f)$ to itself, hence if $g \neq g'$ thus $g\cdot g'=0$, then  $\rho(h)(g) \neq \rho(h)(g')$ and  $\rho(h)(g)\cdot\rho(h)(g')=0$. Moreover
\begin{equation*}
\begin{split}
\rho(h)(g\cdot g) &=\rho(h)\Big(\sum_{k\in G}f(g^{-1}k)k\Big)\\
& =\sum_{k\in G}f(g^{-1}k)\rho(h)(k)\\
& =\sum_{k\in G}f(g^{-1}k) hkh^{-1}\\
& =\sum_{k\in G}f(hg^{-1}kh^{-1}) hkh^{-1}\text{ (since $f$ is a class function)}\\ 
& =\sum_{k\in G}f(hg^{-1}h^{-1}hkh^{-1}) hkh^{-1}\\ 
& =\sum_{k\in G}f(hg^{-1}h^{-1}k) k\\
& =\sum_{k\in G}f(\rho(h)(g)^{-1}k) k\\
& = \rho(h)(g)\cdot\rho(h)(g).
\end{split}
\end{equation*}
Therefore, $\rho(h)$ is an actual automorphism of $\Cay(f)$ for every $h\in G$ and by construction $\rho(h)=\rm{Id}$ if and only if $h\in Z(G)$. Hence $\ker(\rho)=Z(G)$ and $K_2=\Im(\rho)\cong G/Z(G)$.

Finally, notice that every $\rho(h)\in K_2$ fixes, at least, the element $h$ in $B$, while every non-trivial element in $K_1$  acts freely on $B$. Therefore  $K_1\cap K_2=\{1\}$.
\end{proof}

We obtain the following corollary:

\begin{corollary}
Let $\A=\Bbbk[G]$, and $f\colon G\to \Bbbk$ be a class function satisfying that $\Aut(\Cay(f))\cong G$. Then $G$ is abelian.
\end{corollary}

\begin{proof}
According to Proposition \ref{prop:class_function}, as $\Aut(\Cay(f))\cong G$, the subgroup $G/Z(G)$ must be trivial.

\end{proof}

\subsection*{Funding}

\small{This work was partially supported by MCIN/AEI/10.13039/501100011033 
[PID2020-115155GB-I00 and TED2021-131201B-I00 to C.C., PID2020-118452GB-I00 to V.M., 
PID2019-104236GB-I00 to A.T, and PID2020-118753GB-I00 to A.V.], 
and by Junta de Andaluc\'{\i}a [UMA18-FEDERJA-119,  FQM-336  to A.T., and PROYEXCEL-00827, FQM-213 to A.V.].}


\bibliographystyle{abbrv}
\bibliography{CMTV}
\end{document}